\documentclass[11pt,a4paper,reqno]{amsart}
\usepackage{amssymb,amsmath,amsthm}
\usepackage[T1]{fontenc}

\newtheorem{hypo}{Hypothesis}

\newtheorem{prop}[hypo]{Proposition}

\newtheorem{thm}[hypo]{Theorem}

\newtheorem{lem}[hypo]{Lemma}

\newtheorem{coro}[hypo]{Corollary}

\newtheorem{exe}[hypo]{Examples}

\DeclareMathOperator{\cov}{Cov}

\def\A{\mathcal{A}}
\def\B{\mathcal{B}}
\def\C{\mathcal{C}}

\def\PP{\mathbb{P}}
\def\RR{\mathbb{R}}
\def\ZZ{\mathbb{Z}}
\def\CC{\mathbb{C}}

\begin{document}

\title{Quantitative recurrence in two-dimensional extended processes}
\date{\today}
\author{Françoise Pène and Benoît Saussol}
\address{Laboratoire de mathématiques de Brest, CNRS UMR 6205\\ 
6, avenue Victor Le Gorgeu, CS 93837, 29238 BREST Cedex 3, FRANCE}
\email{francoise.pene@univ-brest.fr}
\email{benoit.saussol@univ-brest.fr}

\begin{abstract}
Under some mild condition, a random walk in the plane is recurrent.
In particular each trajectory is dense, and a natural question
is how much time one needs to approach a given small neighborhood of the origin. 
We address this question in the case of some extended dynamical systems similar to planar random walks, including $\ZZ^2$-extension of hyperbolic dynamics.
We define a pointwise recurrence rate and relate it to the dimension of the process, and establish a convergence in distribution
of the rescaled return times near the origin. 
\end{abstract}
\maketitle

\section{Introduction}

\subsection{Motivation}

This work was partly motivated by the recurrence properties of the planar Lorentz process.  
Given an initial condition, say $x$, we thus know that the 
process will return back $\varepsilon$ close to 
its starting point $x$. A basic question is : \emph{when} ? 
For finite measure 
preserving dynamical systems this question has some 
deep relations to the Hausdorff dimension of the 
underlying invariant measure. Namely, if $\tau_\varepsilon(x)$ 
represents this time, in many situations 
\[
\tau_\varepsilon(x)\approx 
\frac{1}{\varepsilon^\mathrm{dim}}
\]
for typical points $x$, 
where $\mathrm{dim}$ is the Hausdorff dimension of the 
underlying invariant measure. This has been proved for example for interval maps \cite{stv} and rapidly mixing systems \cite{bs,saussol}.
Another type of results is the exponential distribution of rescaled return times and the lognormal fluctuations of the return times~\cite{Hirata,cgs}.

In this paper we are dealing with systems where the 
underlying 
natural measure is indeed infinite. 
This typically causes the return times to be non 
integrable, in 
contrast with the finite measure case.
However, the systems we are thinking about 
have in common the 
property that, in some sense, the behaviors at small 
scale and at 
large scale are independent. 
The large scale dynamics being some sort of recurrent random walk, and the small scale dynamics a finite measure preserving system.
Although our first motivation was Lorentz process, we will not mention it further in this paper. We instead provide some results related to processes which are in essence similar to it. The first case treated in Section~\ref{sec:toy} is a toy model designed to give the hint of the general case. Then, in Section~\ref{sec:prw} we briefly mention the case of planar random walks. Finally, in Section~\ref{sec:zx} we give a complete analysis of the quantitative behavior of return times in the case of $\ZZ^2$-extensions of subshifts of finite type.

\subsection{Description of the main result : $\ZZ^2$-extensions of subshifts of finite type}

In this study of the quantitative behavior of recurrence we choose to work with  $\ZZ^2$-extensions of hyperbolic dynamics. We emphasize that this dimension 2 is at the threshold between recurrent and non recurrent processes, since in higher dimension these processes are neither recurrent (except if degenerate).
It makes sense to show how our results behave with respect to the dimension.
For completeness, we call the non-extended system itself a $\ZZ^0$-extension. 
In this non-extended case,
the type of results we have in mind (see Table~\ref{tab:1}) have already been established
respectively by Ornstein and Weiss \cite{ow}, Hirata \cite{Hirata} and Collet, Galves and Schmitt~\cite{cgs}. 
The case of $\ZZ^2$-extension is completely done in Section~\ref{sec:zx}.
The case of $\ZZ^1$-extension can be derived easily following the technique used in the present paper. The essential difference is that the local limit theorem has the one-dimensional scaling in $\frac{1}{\sqrt{n}}$, instead of $\frac{1}{n}$ in the two-dimensional case.
The following table summarizes the different results as the dimension changes. 
The first line of results corresponds to Theorem~\ref{thm:cas3}, the second to Theorem~\ref{thm:cas3bis} and the third to Corollary~\ref{cor:cas3}.
We refer to Section~\ref{sec:zx} for precise statements.

\begin{table}[htbp]
	\centering
		\begin{tabular}{|l|c|c|c|c|}
		\hline
		&&&\\
		Dimension	& 
		$\ZZ^0$-extension & $\ZZ^1$-extension & $\ZZ^2$-extension \\
	&&&\\
		\hline
	&&&\\
		Scale &	$\displaystyle\lim_{\varepsilon\to0}\frac{\log\tau_\varepsilon}{-\log\varepsilon}=d$	&
		$\displaystyle\lim_{\varepsilon\to0} \frac{\log\tau_\varepsilon}{-\log\varepsilon}=\mathbf{2}d$ &
		$\displaystyle\lim_{\varepsilon\to0} \frac{\log\mathbf{log}\tau_\varepsilon}{-\log\varepsilon}=d$ \\
	&&&\\
		\hline
	&&&\\
		Local law &
		$\nu\bigl(\nu(B_\varepsilon)\tau_\varepsilon>t\bigr)\to e^{-t}$ &
		$\nu\bigl(\nu(B_\varepsilon)^\mathbf{2}\tau_\varepsilon>t\bigr)\to \frac{1}{1+\beta\sqrt{t}}$  &
		$\nu\bigl(\nu(B_\varepsilon)\mathbf{log}\tau_\varepsilon>t\bigr)\to \frac{1}{1+\beta t}$ \\
	&&&\\
		 \hline
	
		Lognormal&&& \\
		fluctuations &
		$\varepsilon^d\tau_\varepsilon$ &
		$\varepsilon^{\mathbf{2}d}\tau_\varepsilon$ &
		$\varepsilon^d\mathbf{log}\tau_\varepsilon$ \\
		&&&\\
		\hline
		\end{tabular}

	\caption{Recurrence for $\ZZ^k$-extensions. $B_\varepsilon$ denotes the ball of radius $\varepsilon$, $\nu$ is a Gibbs measure on the base and $d$ is the Hausdorff dimension of $\nu$.}
	\label{tab:1}
\end{table}

\section{A toy model in dimension two}\label{sec:toy}

We present a toy model designed to posses a lot of independence. It has the advantage of giving the right formulas with elementary proofs.

\subsection{Description of the model and statement of the results}

Let us consider two sequences of independent identically distributed 
random variables $(X_n)_{n\ge 1}$
and $(Y_n)_{n\ge 0}$ independent one from the other 
such that~:
\begin{itemize}
\item the random variable $X_1$ is uniformly 
distributed on $\{(1,0),(-1,0),(0,1),(0,-1)\}$;
\item the random variable $Y_0$ is uniformly 
distributed on $]0;1[^2$.
\end{itemize}
Let us notice that $S_n:=\sum_{k=1}^nX_k$ 
(with the convention $S_0:=0$)
is the symmetric random walk on $\ZZ^2$.
We study a kind of random walk $M_n$ on $\RR^2$ given by $M_n=S_n+Y_n$.

Another representation of our model could be the following. Let $S=\RR^d$ and consider the system
$\ZZ^2\times S$. Attached to each site $i\in\ZZ^2$ of the lattice, there is a local system which lives on $S$ and $\sigma_n$ is a i.i.d. sequence of $S$-valued random variable with some density $\rho$, independent of the $X_n$'s. Then we look at the random walk $(S_n,\sigma_n)$, thinking at $\sigma_n$ as a spin. 

We want to study the asymptotic behavior,
as $\varepsilon$ goes to zero, of the return time in 
the open ball $B(M_0,\varepsilon)$ 
of radius $\varepsilon$ centered at $x$ 
(for the euclidean metric).
Let 
$$\tau_\varepsilon:=\min\{m\ge1\colon |M_m-M_0|
    <\varepsilon\}.$$
Note that, for all $x\in[0;1[^2$, we have~: 
$\tau_\varepsilon=\min\{m\ge 1\colon
    \vert M_m - x\vert<\varepsilon \}$.
We will prove the following~:
\begin{thm}\label{thm:pointwise}
Almost surely, $\displaystyle\frac{\log\log 
\tau_\varepsilon}{-\log \varepsilon }$
converges to the dimension 2 of the Lebesgue measure 
on $\RR^2$ as $\varepsilon$ goes to zero.
\end{thm}
\begin{thm}\label{thm:law}
For all $t\ge0$ we have~:
$$
\lim_{\varepsilon\rightarrow 0}\PP\left(
     \lambda(B(M_0,\varepsilon))
     \log \tau_\varepsilon\le t\right)=\frac{1}{\displaystyle 1+\frac\pi t} .
$$
\end{thm}

\subsection{Proof of the pointwise convergence of the recurrence rate to the dimension} 

To simplify the exposition we suppose that $M_0$ is 
in $]0;1[^2$ and that 
$\varepsilon$ is so small that $B(x,\varepsilon)$ 
is contained in $[0;1[^2$.

First, let us define $R_{1}:=\min\{m\ge 1\colon
   S_m=0\}$.
According to \cite{DvoretskiErdos}, 
we know that we have~:
\begin{equation}\label{eq:R1}
\PP (R_1>s)\sim \frac{\pi}{\log s}
\quad \text{as $s$ goes to infinity}
\end{equation}
We then define for any $p\ge0$ the $p^{th}$ return 
time $R_p$ in $[0;1[^2$ by induction~:
$$R_{p+1}:=\inf\{m > R_{p}\ :\ S_m=0\}. $$
Observe that $R_p$ is the $p^{th}$ return time at the origin of the random walk $S_n$ on the lattice, thus the delays between successive return times $R_p-R_{p-1}$, setting $R_0=0$, are independent and identically distributed. Consequently~:
\begin{equation}\label{eq:F1}
\PP (R_p-R_{p-1}>s) = \PP(R_1>s)
\end{equation}

The proof of Theorem~\ref{thm:pointwise} follows from these two lemmas

\begin{lem}\label{lem:Rn}
Almost surely, $\displaystyle \frac{\log\log R_n}{\log n}\to 1$ as $n\to\infty$.
\end{lem}
\begin{proof}
It suffices to prove that for any $0<\alpha<1$, almost surely, $e^{n^{1-\alpha}}\le R_n\le ne^{n^{1+\alpha}}$ provided $n$ is sufficiently large.
By independence and equation \eqref{eq:F1} we have
$$
\PP(\log R_n\le n^{1-\alpha})
\le\PP (\forall p\le n,\log (R_p-R_{p-1})\le n^{1-\alpha})
=\PP(\log R_1\le n^{1-\alpha})^n.
$$
According to the asymptotic formula~\eqref{eq:R1}, for $n$ sufficiently large
$$
\PP(\log R_1\le n^{1-\alpha})^n \le \left(1-{\frac{\pi}{2n^{1-\alpha}}}\right)^n
    \le e^{-{\pi \frac{n^\alpha}{2}}}.
$$
The first inequality follows then from Borel Cantelli lemma.

Moreover, according to formulas \eqref{eq:F1} and \eqref{eq:R1}, we have
$\sum_{n\ge 1}\PP (\log(R_n-R_{n-1})>n^{1+\alpha}))<+\infty $.
Hence, by Borel Cantelli lemma, we know that
almost surely, for $n$ sufficiently large, we have $R_n-R_{n-1}\le e^{n^{1+\alpha}}$.
{}From this we get the second inequality.
\end{proof}

Observe that $\tau_\varepsilon=R_{T_\varepsilon}$
with $T_{\varepsilon}:=\min\{\ell\ge 1\ :\ 
| Y_{R_\ell}-Y_0 | < \varepsilon \}$.
\begin{lem}\label{lem:Tepsilon}
Almost surely, $\frac{\log T_{\varepsilon}}{ -
\log \lambda(B(Y_0,\varepsilon))
} \to 1$ as $\varepsilon\to0$.
\end{lem}

\begin{proof}
By independence of the $Y_\ell$, 
the random variable $T_\varepsilon$ has a geometric 
distribution
with parameter $\lambda_\varepsilon:=\lambda(B(Y_0,\varepsilon))
=\pi\varepsilon^2$.
For any $\alpha>0$ we have the simple decomposition
$$
\PP\left(\left\vert
\frac{ \log T_{\varepsilon}}{ -\log\lambda_\varepsilon}-1\right\vert>\alpha\right)=
\PP\left(T_{\varepsilon}>\lambda_\varepsilon^{-1-\alpha}\right)
+\PP\left(T_{\varepsilon}<\lambda_\varepsilon^{-1+\alpha}\right).\\
$$
The first term is directly handle by Markov inequality~:
$$\PP\left(T_{\varepsilon}>\lambda_\varepsilon^{-1-\alpha}\right)
    \le \lambda_\varepsilon^{\alpha},$$
while the second term may be computed using the geometric distribution ~:
\begin{eqnarray*}
{\Bbb P}(T_\varepsilon<\lambda_\varepsilon^{-1+\alpha})
&=& 1-(1-\lambda_\varepsilon)^{(\lambda_\varepsilon^{-1+\alpha})}\\
&=& 1-\exp\left[ \lambda_\varepsilon^{-1+\alpha}\log(1-\lambda_\varepsilon)\right]\\
&\le& -\lambda_\varepsilon^{-1+\alpha}\log(1-\lambda_\varepsilon)\\
&=& O(\lambda_\varepsilon^{\alpha}).
\end{eqnarray*}
Let us define $\varepsilon_n:= n^{-{1/\alpha}}$.
According to the Borel-Cantelli lemma, 
$ \lambda_{\varepsilon_n}
T_{\varepsilon_n}$
converges almost surely to the constant~1. 
The conclusion follows from the facts
that $(\varepsilon_n)_{n\ge 1}$ is a decreasing sequence
of real numbers satisfying $\lim_{n\rightarrow +\infty}\varepsilon_n=0$ and $\lim_{n\rightarrow +\infty}{\varepsilon_n\over\varepsilon_{n+1}}=1$,
and $T_\varepsilon$ is monotone in $\varepsilon$.
\end{proof}

\begin{proof}[Proof of Theorem~\ref{thm:pointwise}]
The theorem follows from Lemma~\ref{lem:Rn} and 
Lemma~\ref{lem:Tepsilon} since
$$
\frac{\log \log \tau_\varepsilon}{-\log\varepsilon}
=\frac{\log\log R_{T_\varepsilon}}
{\log T_\varepsilon} 
\frac{\log T_\varepsilon}{-\log\lambda_\varepsilon} 
\frac{\log\lambda_\varepsilon}{\log \varepsilon}
\to 1\times 1\times 2
$$
almost surely as $\varepsilon\to0$.
\end{proof}

\subsection{Proof of the convergence in distribution of the rescaled return time}

\begin{proof}[Proof of Theorem~\ref{thm:law}]
Let $t>0$. By independence of $T_\varepsilon$ and the $R_n$ we have 
$$
F_\varepsilon(t):=
\PP (\lambda(B(Y_0,\varepsilon)\log\tau_\varepsilon
    \le t)=
\sum_{n\ge 1}
\PP( T_\varepsilon=n)\PP\left(\log R_n \le 
  \frac t{\lambda_\varepsilon}\right).
$$
Since $T_\varepsilon$ has a geometric law 
with parameter $\lambda_\varepsilon$, 
$F_\varepsilon(t)$ is equal to $
G_{\lambda_\varepsilon}(t)$ with~:
$$
G_\delta(t) :=\sum_{n\ge1}\delta(1-\delta)^{n-1} \PP (\log R_n\le\frac t\delta).
$$
First, we notice that the independence of the successive returns gives for any $u>0$ that 
\[
\PP \left(R_n\le u\right)
\le
\PP \left(\max_{k=1,...,n} R_k-R_{k-1}\le u\right)
=
\PP\left(R_1\le u\right)^n
\]
Let $\alpha<1$. 
Using the inequality above and the equivalence relation~\eqref{eq:R1} we get that for any $\delta>0$ sufficiently small, 
\[
G_\delta(t) \le \sum_{n\ge1} \delta(1-\delta)^{n-1}\left(1-\alpha\frac{\pi\delta}{t}\right)^n
=\frac{1}{\displaystyle 1+\alpha\frac{\pi}{t}}+O(\delta).
\]
This implies that $\limsup_{\delta\to0} F_\epsilon(t)\le \frac{1}{1+\frac{\pi}{t}}$.

Fix $A>0$ and keeping the same notations observe that we have
$
F_\varepsilon(t) \ge H_{\lambda_\varepsilon
}(t)$
with~:
\[
H_\delta(t) := \sum_{1\le n\le A/\delta} \delta(1-\delta)^{n-1}\PP\left(\log R_n\le\frac{t}{\delta}\right).
\]
Note that the independence gives in addition that for any $u>0$
\[
\PP \left(R_n\le u\right)
\ge
\PP \left(\max_{k=1,...,n} R_k-R_{k-1}\le u/n\right)
=
\PP\left(R_1\le u/n\right)^n.
\]
Let $\alpha>1$.
Using the inequality above and the equivalence relation~\eqref{eq:R1} we get that for sufficiently small $\delta>0$
\[
H_\delta(t) 
\ge \sum_{1\le n\le A/\delta} \delta(1-\delta)^{n-1}\left(1-\alpha\frac{\pi}{\frac{t}{\delta}-\log n}\right)^n 
\ge \sum_{1\le n\le A/\delta} \delta(1-\delta)^{n-1}
\left(1-\alpha^2\frac{\pi\delta}{t}\right)^n.
\]
Evaluating the limit when $\delta\to0$ of the geometric sum and then letting $A\to\infty$ we end up with $\liminf_{\delta\to0}H_\delta(t)\ge \frac{1}{1+\frac{\pi}{t}}$.
\end{proof}

\section{Random walk on the plane}\label{sec:prw}

We now consider a true random walk on $\RR^2$, $M_n=X_1+\cdots+X_n$ where 
the $X_i$'s are i.i.d. random variables distributed with a law $\mu$ of 
zero mean, with invertible finite covariance matrix $\Sigma^2$ and characteristic function
$\hat\mu(t)=\int e^{it\cdot x}d\mu(x)$.
Let $\tau_\varepsilon$ be the minimal time for 
the walk to return 
in the $\varepsilon$-neighborhood of the origin~:
$$\tau_\varepsilon:=\min\{n\ge 1\colon \vert M_n\vert
    <\varepsilon\} .$$
\begin{thm}\label{thm:RWP}
Assume additionally that the distribution $\mu$ satisfies the Cramer condition 
\[
\limsup_{|t|\to\infty}|\hat\mu(t)|<1.
\]
Then almost surely $\lim_{\varepsilon\to0}\frac{\log\log\tau_\varepsilon}{-\log\varepsilon}=
2$.
\end{thm}

We remark that a kind of Cramer's condition on the law is necessary, since there exists some planar recurrent random walks for which the statement of the theorem is false (the return time being even larger than expected). We discovered after the completion of the proof of this theorem that its statement is contained in Theorem~2 of 
Cheliotis's recent paper~\cite{Cheliotis}. For completeness we describe the strategy of our original proof here but leave most details to the reader. A key point is a uniform version of the local limit theorem. 
Indeed we need an estimation of the type $\PP ( |S_n|< \varepsilon ) \sim \frac{c\varepsilon^2}{n}$,
with some uniformity in $\varepsilon$ (for some $c>0$). 
One can follow the classical proof of the
local limit theorem 
(see Theorem 10.17 of~\cite{Breiman}) to get the 
following~:
\begin{lem}
Let $\delta\in]0;1/2[$.
There exists $c_1>0$, $c_2>0$, $\varepsilon_0>0$, $a>0$ and an integer $N$,
such that, for any $n>N$, for any $\varepsilon\in]0;1[$ ~: 
 $${c_1\varepsilon^2\over n}-{
    \exp(-an^{1-2\delta})\over\varepsilon}\le 
    \PP\left( S_n\in B(0,\varepsilon)\right)\le 
 {c_2\varepsilon^2\over n}+{
    \exp(-an^{1-2\delta})\over\varepsilon}. $$
\end{lem}

Then, this information on the probability of return is strong enough to estimate the first return time to the $\varepsilon$-neighborhood of the origin. 

\begin{proof}[Proof of Theorem \ref{thm:RWP}]
For any $\alpha>\frac12$, using $\varepsilon_n = 1/\log^{\alpha} n$, we get that
$\PP(|S_n|<\varepsilon_n)$ is summable. By the Borel Cantelli lemma, we have $\tau_{\varepsilon_n}(x)>n$ eventually almost surely. Thus
\[
\liminf_{n\to\infty}\frac{\log\log\tau_{\varepsilon_n}}{-\log\varepsilon_n}
\ge \liminf_{n\to\infty} \frac{\log\log n}{\log\log^\alpha n}=\frac1\alpha,
\]
which implies by monotonicity and the fact that $\alpha$ is arbitrary that $\liminf_{\epsilon\to0} \frac{\log\log\tau_\varepsilon}{-\log\varepsilon} \ge 2$.

Let $\alpha<{1\over 2}$.
To control the $\limsup$, we will take $n=n_\varepsilon=\lceil
 \varepsilon^{-{1\over\alpha}}\rceil$.
We use a similar decomposition to that of Dvoretski and Erd\"os in 
\cite{DvoretskiErdos}. Let 
$A_k=\{|S_k|<\varepsilon \text{ and } \forall p=k+1,\ldots,n, |S_p-S_k|>2\varepsilon\}$.
The $A_k$'s are disjoint, hence by independence and invariance,
and with our choice for $n$~:
\[
1 \ge \sum_{k=1}^n \PP(A_k) 
\ge \sum_{k=1}^n\PP(|S_k|<\varepsilon)\PP(\tau_{2\varepsilon}>n-k)
\ge \sum_{k=1}^n \PP(|S_k|<\varepsilon) \PP(\tau_{2\varepsilon}>n)
\ge \sum_{k=N}^n {c\varepsilon^2\over k} \PP(\tau_{2\varepsilon}>n).
\]
Hence we have $\PP(\tau_{\varepsilon}>n)\le \frac{1}{c\varepsilon^2\log n}
  \le c\varepsilon^{{1\over\alpha}-2}$,
if $n$ is large enough. 
Let $\varepsilon_p=p^{\frac{-2}{\alpha-2}}$. By Borel Cantelli lemma we have
$\log\tau_{\varepsilon_p}\le \varepsilon_p^{-\alpha}$ eventually almost surely, hence $\limsup_{p\to\infty}\frac{\log\log\tau_{\varepsilon_p}}{-\log\varepsilon_p}\le\alpha$.
By monotonicity and the fact that $\alpha$ is arbitrary we get the result.
\end{proof}

\section{Case of Euclidean extension of hyperbolic systems}\label{sec:zx}

\subsection{Description of the ${\Bbb Z}^2$-extension of a mixing subshift}

Let us fix a finite set $\A$ called alphabet.
Let us consider a matrix $M$ indexed by ${\A}\times{\A}$ with 0-1 entries.
We suppose that there exists a positive integer $n_0$ such that each component of 
$M^{n_0}$ is non zero.
We define the set of allowed sequences $\Sigma$ as follows
$$\Sigma:=\{\omega:=(\omega_n)_{n\in\Bbb Z}\ :\ 
\forall n\in{\Bbb Z},\ M(\omega_n,\omega_{n+1})=1\}. $$
We endow $\Sigma$ with the metric $d$ given by
$$d\left(\omega,\omega'\right):=
e^{-m},
$$
where $m$ is the greatest integer such that $\omega_i=\omega_i'$ whenever $|i|<m$.
We define the shift $\theta:\Sigma\rightarrow\Sigma$ by
$\theta\left((\omega_n)_{n\in\Bbb Z}\right)=(\omega_{n+1})_{n\in\Bbb Z}$.
For any function $f\colon \Sigma\to\RR$ we 
denote by $S_n f=\sum_{\ell=0}^{n-1} f
\circ\theta^\ell$ its ergodic sum.
Let us consider an Hölder continuous function $\varphi:\Sigma\rightarrow{\Bbb Z}^2$. 
We define the $\ZZ^2$-extension $F$ of the shift $\theta$ by 
\[
\begin{split}
F\colon \Sigma\times \ZZ^2 &\to \Sigma\times\ZZ^2\\
(x,m) &\mapsto (\theta x,m+\varphi(x)).
\end{split}
\]

\begin{figure}[h]\label{fig:extension}
\def\JPicScale{0.7}
\ifx\JPicScale\undefined\def\JPicScale{1}\fi
\unitlength \JPicScale mm
\begin{picture}(150,90)(0,0)
\linethickness{0.1mm}
\put(10,90){\line(1,0){20}}
\linethickness{0.1mm}
\put(10,70){\line(0,1){20}}
\linethickness{0.1mm}
\put(30,70){\line(0,1){20}}
\linethickness{0.1mm}
\put(10,70){\line(1,0){20}}
\linethickness{0.1mm}
\put(50,90){\line(1,0){20}}
\linethickness{0.1mm}
\put(50,70){\line(0,1){20}}
\linethickness{0.1mm}
\put(70,70){\line(0,1){20}}
\linethickness{0.1mm}
\put(50,70){\line(1,0){20}}
\linethickness{0.1mm}
\put(90,90){\line(1,0){20}}
\linethickness{0.1mm}
\put(90,70){\line(0,1){20}}
\linethickness{0.1mm}
\put(110,70){\line(0,1){20}}
\linethickness{0.1mm}
\put(90,70){\line(1,0){20}}
\linethickness{0.1mm}
\put(130,90){\line(1,0){20}}
\linethickness{0.1mm}
\put(130,70){\line(0,1){20}}
\linethickness{0.1mm}
\put(150,70){\line(0,1){20}}
\linethickness{0.1mm}
\put(130,70){\line(1,0){20}}
\linethickness{0.1mm}
\put(130,60){\line(1,0){20}}
\linethickness{0.1mm}
\put(130,40){\line(0,1){20}}
\linethickness{0.1mm}
\put(150,40){\line(0,1){20}}
\linethickness{0.1mm}
\put(130,40){\line(1,0){20}}
\linethickness{0.1mm}
\put(90,60){\line(1,0){20}}
\linethickness{0.1mm}
\put(90,40){\line(0,1){20}}
\linethickness{0.1mm}
\put(110,40){\line(0,1){20}}
\linethickness{0.1mm}
\put(90,40){\line(1,0){20}}
\linethickness{0.1mm}
\put(50,60){\line(1,0){20}}
\linethickness{0.1mm}
\put(50,40){\line(0,1){20}}
\linethickness{0.1mm}
\put(70,40){\line(0,1){20}}
\linethickness{0.1mm}
\put(50,40){\line(1,0){20}}
\linethickness{0.1mm}
\put(10,60){\line(1,0){20}}
\linethickness{0.1mm}
\put(10,40){\line(0,1){20}}
\linethickness{0.1mm}
\put(30,40){\line(0,1){20}}
\linethickness{0.1mm}
\put(10,40){\line(1,0){20}}
\linethickness{0.1mm}
\put(10,10){\line(1,0){20}}
\linethickness{0.1mm}
\put(10,10){\line(0,1){20}}
\linethickness{0.1mm}
\put(30,10){\line(0,1){20}}
\linethickness{0.1mm}
\put(10,30){\line(1,0){20}}
\linethickness{0.1mm}
\put(50,30){\line(1,0){20}}
\linethickness{0.1mm}
\put(50,10){\line(0,1){20}}
\linethickness{0.1mm}
\put(70,10){\line(0,1){20}}
\linethickness{0.1mm}
\put(50,10){\line(1,0){20}}
\linethickness{0.1mm}
\put(90,30){\line(1,0){20}}
\linethickness{0.1mm}
\put(90,10){\line(0,1){20}}
\linethickness{0.1mm}
\put(110,10){\line(0,1){20}}
\linethickness{0.1mm}
\put(90,10){\line(1,0){20}}
\linethickness{0.1mm}
\put(130,30){\line(1,0){20}}
\linethickness{0.1mm}
\put(130,10){\line(0,1){20}}
\linethickness{0.1mm}
\put(150,10){\line(0,1){20}}
\linethickness{0.1mm}
\put(130,10){\line(1,0){20}}
\linethickness{0.1mm}
\put(17.5,15){\circle{7.07}}

\linethickness{0.1mm}
\put(17.5,15){\circle{0.25}}

\linethickness{0.1mm}
\qbezier(17.5,15)(22.69,23.48)(27.83,29.82)
\qbezier(27.83,29.82)(32.98,36.16)(38.09,40.38)
\linethickness{0.1mm}
\qbezier(38.09,40.38)(43.2,44.59)(48.68,47)
\qbezier(48.68,47)(54.16,49.4)(60,50)
\linethickness{0.5mm}
\put(60,50){\line(1,0){0.12}}
\put(60.12,50){\vector(1,0){0.12}}
\linethickness{0.1mm}
\qbezier(60,50)(66.43,40.21)(74.76,33.21)
\qbezier(74.76,33.21)(83.08,26.21)(93.31,22)
\linethickness{0.1mm}
\qbezier(93.31,22)(103.54,17.79)(116.46,16.04)
\qbezier(116.46,16.04)(129.38,14.29)(145,15)
\linethickness{0.1mm}
\put(145,15){\line(1,0){0.12}}
\put(145.12,15){\vector(1,0){0.12}}
\linethickness{0.1mm}
\qbezier(145,15)(134.58,22.77)(126.36,30.86)
\qbezier(126.36,30.86)(118.14,38.96)(112.12,47.38)
\linethickness{0.1mm}
\qbezier(112.12,47.38)(106.1,55.8)(101.82,65.2)
\qbezier(101.82,65.2)(97.54,74.61)(95,85)
\linethickness{0.1mm}
\multiput(94.88,85.48)(0.12,-0.48){1}{\line(0,-1){0.48}}
\put(94.88,85.48){\vector(-1,4){0.12}}
\linethickness{0.1mm}
\qbezier(95,85)(101.49,83.04)(107.59,81.9)
\qbezier(107.59,81.9)(113.7,80.77)(119.41,80.47)
\linethickness{0.1mm}
\qbezier(119.41,80.47)(125.12,80.17)(130.9,80.68)
\qbezier(130.9,80.68)(136.67,81.18)(142.5,82.5)
\linethickness{0.1mm}
\multiput(142.5,82.5)(0.12,0.03){1}{\line(1,0){0.12}}
\put(142.62,82.53){\vector(4,1){0.12}}
\put(20,12.5){\makebox(0,0)[cc]{$(x,m)$}}

\put(62.5,52.5){\makebox(0,0)[cc]{$F(x,m)$}}

\put(20,12.5){\makebox(0,0)[cc]{}}

\put(145,20){\makebox(0,0)[cc]{$F^2(x,m)$}}

\put(92.5,87.5){\makebox(0,0)[cc]{$F^3(x,m)$}}

\put(145,82.5){\makebox(0,0)[cc]{$F^4(x,m)$}}

\put(90,87.5){\makebox(0,0)[cc]{}}

\put(35,35){\makebox(0,0)[cc]{}}

\put(35,35){\makebox(0,0)[cc]{$\varphi(x)=(1,1)$}}

\put(80,32.5){\makebox(0,0)[cc]{$\varphi(\theta x)=(2,-1)$}}

\end{picture}
\caption{Dynamics of the $\ZZ^2$-extension $F$ of the shift.}
\end{figure}
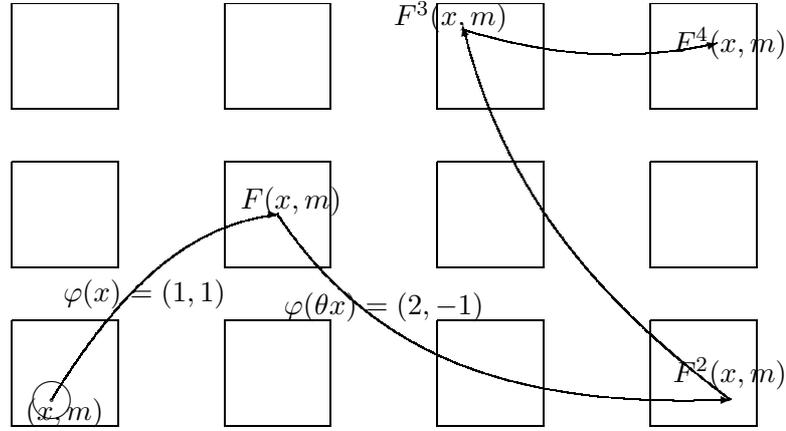

We want to know the time needed for a typical orbit starting at $(x,m)\in \Sigma\times \ZZ^2$ to return $\varepsilon$-close to the initial point after iterations of the map $F$.
By translation invariance we can assume that the orbit starts in the cell $m=0$.
More precisely, let
\[
\tau_\varepsilon(x)=\min\{n\ge1\colon F^n(x,0)\in B(x,\varepsilon)\times\{0\}\}.
\]
Observe that $F^n(x,m)=(\theta^nx,m+S_n\varphi(x))$, thus
\[
\tau_\varepsilon(x)=\min\{n\ge1\colon S_n\varphi(x)=0\text{ and } d(\theta^nx,x)<\varepsilon)\}.
\]

Let $\nu$ be the Gibbs measure associated to some Hölder continuous potential $h$, and denote by $\sigma^2_h$ the asymptotic variance of $h$ under the measure $\nu$. Recall that $\sigma^2_h$ vanishes if and only if $h$ is  cohomologous to a constant, and in this case $\nu$ is the unique measure of maximal entropy.

We know that there exists a positive integer $m_0$
such that the function $\varphi$ is constant on each $m_0$-cylinders.

Let us denote by $\sigma_\varphi^2$ the asymptotic covariance 
matrix of $\varphi$~:
$$\sigma_\varphi^2=\lim_{n\rightarrow +\infty}\cov_\nu\left(\frac{1}{\sqrt{n}}S_n\varphi\right).$$
{\bf We suppose that $\sigma_\varphi^2$ is invertible.} Note that if it is not the case then it means that range of $S_n\varphi$ is essentially contained in a one-dimensional lattice; in this case we can reduce our study to the corresponding $\ZZ$-extension.

{\bf We add the following hypothesis of non-arithmeticity on $\varphi$~:
We suppose that, for any $u\in[-\pi;\pi]^2\setminus\{
 (0,0)\}$ the only solutions $(\lambda,g)$, with $\lambda\in\CC$ and $g\colon\Sigma\to\CC$ measurable with $|g|=1$, of the functional equation
\[
g\circ\sigma \overline g = \lambda e^{iu\cdot\varphi}
\]
is the trivial one $\lambda=1$ and $g=const$.} Note that if it is not the case then it should mean that the range of $S_n\varphi$ is essentially contained in a sub-lattice; in this case we could just do a change of basis and apply our result to the new twisted $\ZZ^2$-extension.
We emphasize that this non-arithmeticity condition is equivalent to the fact that all the circle extensions $T_u$ defined by $T_u(x,t)=(\theta(x),t+u\cdot\varphi(x))$ are weakly-mixing for $u\in[-\pi;\pi]^2\setminus\{ (0,0)\}$.

In this context, we prove the following results~:
\begin{thm}\label{thm:cas3}
The sequence of random variables 
${\log\log\tau_\varepsilon\over-\log\varepsilon}$
converges almost surely as $\varepsilon\to0$ to the Hausdorff dimension $d$ of the measure $\nu$.
\end{thm}
\begin{thm}\label{thm:cas3bis}
The sequence of random variables
$\nu(B_\varepsilon(\cdot))\log\tau_{\varepsilon}(\cdot)$ converges in distribution as $\varepsilon\to0$ to a random
variable with distribution function of density $t\mapsto
   {\beta t\over 1+\beta t}{\bf 1}_{(0;+\infty)}(t)$, 
with $\beta:={1\over 2\pi\sqrt {\det\sigma_\varphi^2}}$.
\end{thm}
\begin{coro}\label{cor:cas3}
If the measure $\nu$ is not the measure of maximal entropy, then the sequence of random variables
${\log\log\tau_\varepsilon
     +d\log \varepsilon\over\sqrt{-\log\varepsilon}}$ converges in distribution as $\varepsilon\to0$ to a centered
gaussian random variable of variance $2\sigma^2_h$.

In the case where $\nu$ is the measure of maximal entropy, then the sequence of random variables $\varepsilon^d\log\tau_{\varepsilon}$ converges in distribution to a finite mixture of the law found in the previous theorem, i.e. there exists some probability vector $\alpha=(\alpha_n)$ and positive constants $\beta_n$ such that
the sequence of random variables $\varepsilon^d\log\tau_\varepsilon$ converges in distribution to a random variable with distribution function of density $\sum_n\alpha_n\frac{\beta_n t}{1+\beta_n t}{\bf 1}_{(0;+\infty)}(t)$.   
\end{coro}

\begin{exe} We provide an example where the function $\varphi(x)$ only depends on the first coordinate $x_0$, i.e. $\varphi(x)=\varphi(x_0)$.
On the shift $\Sigma=\{E,N,W,S\}^\ZZ$ the function $\varphi(E)=(1,0)$, $\varphi(N)=(0,1)$, $\varphi(W)=(-1,0)$ and $\varphi(S)=(0,-1)$ fulfill the hypotheses.
\end{exe}

The rest of the section is devoted to the proof of these results. In Subsection~\ref{sub:zxspectral} we recall some preliminary results and prove a uniform conditional local limit theorem. In Subsection~\ref{sub:zxpointwise} we prove Theorem~\ref{thm:cas3} and in Subsection~\ref{sub:zxfluctuation} we prove Theorem~\ref{thm:cas3bis} and Corollary~\ref{cor:cas3}.

\subsection{Spectral analysis of the Perron-Frobenius operator and Local Limit Theorem}\label{sub:zxspectral}

In order to exploit the spectral properties of the Perron-Frobenius operator we 
quotient out the "past". We define~:
$$\hat\Sigma:=\{\omega:=(\omega_n)_{n\in\Bbb N}\ :\ 
\forall n\in{\Bbb N},\ M(\omega_n,\omega_{n+1})=1\}, $$
$$\hat d\left((\omega_n)_{n\ge 0},(\omega'_n)
_{n\ge 0}\right):=e^{-\hat r(\omega,\omega')}$$
with $\hat r((\omega_n)_{n\ge 0},(\omega'_n)
_{n\ge 0})=\inf\{m\ge 0\ :\ 
\omega_m\ne\omega'_m\}$ and
$$\hat\theta\left((\omega_n)_{n\ge 0}\right)
=(\omega_{n+1})_{n\ge 0}.$$
Let us define the canonical projection $\Pi:\Sigma\rightarrow 
\hat\Sigma$
defined by $\pi\left((\omega_n)_{n\in \Bbb Z}\right)=(\omega_n)_{n\ge 0}$.
Let $\hat\nu$ be the image probability measure (on $\hat\Sigma$) of $\nu$ by $\Pi$.
There exists a function 
$\psi:\hat\Sigma\rightarrow {\Bbb Z}^2$ such that 
$\psi\circ\Pi=\varphi\circ\theta^{m_0}$.

Let us denote by $P:L^2(\hat\nu)\rightarrow 
L^2(\hat\nu)$ the Perron-Frobenius operator
such that~:
$$\forall f,g\in L^2(\hat\nu),\ \ 
\int_{\hat\Sigma}Pf(x)g(x)\, d\hat\nu(x)=
        \int_{\Sigma^+}f(x)g\circ \hat\theta(x)\, 
    d\hat\nu(x).$$
Let $\eta\in]0;1[$. Let us denote by ${\B}$ the set of bounded
$\eta$-Hölder continuous function $g:\hat\Sigma
\rightarrow\mathbb C$ endowed with the usual Hölder
norm~:
$$\Vert g\Vert_{\B}:=\Vert g\Vert_\infty+
\sup_{x\ne y}{\vert g(y)-g(x)\vert\over \hat d(x,y)^\eta}. $$
We denote by ${\B}^*$ the topological dual of 
$\B$.
For all $u\in{\Bbb R}^2$,
we consider the operator $P_u$ defined on $({\B},\Vert\cdot\Vert_{\B})$ by~:
$$P_u(f):=P(e^{iu\psi}f).$$

Note that the hypothesis of non-arithmeticity of $\varphi$ is equivalent to the following one on $\psi$~: \emph{for any $u\in[-\pi;\pi]^2\setminus\{
 (0,0)\}$, the operator $P_u$ has no eigenvalue on the
unit circle.}

We will use the method introduced by Nagaev in 
\cite{Nagaev57,Nagaev61}, adapted by Guivarch and Hardy
in \cite{GuivarchHardy} and extended
by Hennion and Herv\'e in \cite{HennionHerve}.
It is based on the family of operators $(P_u)_u$ 
and their spectral properties expressed in these two propositions.

\begin{prop}[Uniform contraction]
\label{prop:contract} 
There exists $\alpha\in(0;1)$, $C>0$ such that,
for all $u\in[-\pi;\pi]^2\setminus[-\beta;\beta]^2$
and all integer $n\ge 0$, for all $f\in{\B}$,
we have~:
$$\Vert P_u^n(f)\Vert_{\B}\le 
    C\alpha^n \Vert f\Vert_{\B}.$$
\end{prop}
This property easily follows from the fact that the spectral radius 
is smaller than 1 for $u\neq 0$.
In addition, since $P$ is a quasicompact operator on
 $\B$ and since $u\mapsto P_u$ is a regular
perturbation of $P_0=P$, we have~:
\begin{prop}[Perturbation result]
\label{prop:perturb}
There exists $\alpha>0$, $\beta>0$, $C>0$, $c_1>0$, $\theta\in]0;1[$ such that~:
there exists $u\mapsto \lambda_u$ belonging to 
$C^3([-\beta;\beta]^2\rightarrow {\Bbb C})$, there exists $u\mapsto v_u$ 
belonging to $C^3([-\beta;\beta]^2\rightarrow{\B})$, there exists
$u\mapsto \phi_u$ belonging to $C^3([-\beta;\beta]^2\rightarrow{\B}*)$
such that, for all $u\in[-\beta;\beta]^2$, for all $f\in\B$
and for all $n\ge 0$, we have the decomposition~:
$${P_u}^n(f)={\lambda_u}^n\phi_u(f)v_u+N_u^n(f),$$
with 
\begin{enumerate}
\item\label{en:rest} $\Vert{N_u}^n(f)\Vert_{\B}\le C\alpha^n\Vert f\Vert_{\B}$,
\item\label{en:bound} $\vert\lambda_u\vert\le e^{-c_1  |u|^2}$
and $c_1 |u|^2\le  \sigma^2_\phi u\cdot u$,
\item\label{en:initial}
with initial values~:
$v_0={\bf 1}$, $\phi_0=\hat\nu$, $\nabla\lambda_{u=0}=0$ and
$D^2\lambda_{u=0}=-\sigma^2_\varphi$.
\end{enumerate}
\end{prop}
This result is a multidimensional version
of IV-8, IV-11, IV-12 of \cite{HennionHerve}, in this context.

Next proposition is essential to our work. It may be viewed as a doubly local version of the central limit theorem :
first, it is local in the sense that we are looking at the probability that $S_n\varphi=0$
while the classical central limit theorem is only concerned with the probability that $|S_n\varphi|\le \varepsilon \sqrt{n}$~;
second, it is local in the sense that we are looking at this probability conditioned to the fact that we are starting from a set $A$ and landing at a set $B$ on the base.

\begin{prop}\label{prop:pro2}
There exists a real number $C_1>0$ such that, 
for all integer $n> k> m_0$
and all $q>0$, all
$q$-cylinder $A$ of $\Sigma$
and all measurable subset $B$ of $\Sigma^+$, we have~:
$$\left\vert \nu\left(A\cap\left\{S_n\varphi=0\right\} 
 \cap \theta^{-n}(\theta^k(\Pi^{-1}(B)))\right) 
   -{\nu(A)\hat\nu(B)\over 2\pi(n-k)\sqrt{det(\sigma^2_\varphi)} 
  }\right\vert\le C_1
{\hat\nu(B)ke^{\eta q}\over (n-k)^{3/2}}.$$
\end{prop}

\begin{proof}
We want to estimate the measure of the set
$Q=A\cap\{S_n\varphi=0\}\cap \theta^{-n}(\theta^k\Pi^{-1}B)$.
Since $A$ is a $q$-cylinder, $\theta^{-q}A=\Pi^{-1}\hat A$ for the cylinder set $\hat A=\Pi\theta^{-q}A$. 
Next, since $\varphi\circ \theta^{m_0}=\psi\circ\Pi$ we have the identity
$\{S_n\varphi\circ\theta^{m_0}=0\}=\{S_n\psi\circ\Pi=0\}$. Thus using the semi-conjugacy $\hat\theta\circ\Pi=\Pi\circ\theta$
\[
\begin{split}
\theta^{-q-m_0}Q &=\theta^{-m_0}(\Pi^{-1}\hat A) \cap \{S_n\psi\circ\Pi\circ\theta^q=0\}\cap \theta^{-n-q+(k-m_0)}(\Pi^{-1}B))  \\
&=\Pi^{-1}( \hat\theta^{-m_0}(\hat A) \cap \{S_n\psi\circ \hat\theta^q=0\}\cap \hat\theta^{-n-q+(k-m_0)}(B))
\end{split}
\]
Since $\psi$ is integer-valued, the relation $1_{\{k=0\}}=\frac{1}{(2\pi)^2}\int_{[-\pi,\pi]^2}e^{iu\cdot k}du$ for any $k\in\ZZ^2$ gives, by invariance of $\nu$,
\[
\begin{split}
\nu(Q)&={\Bbb E}_{\hat\nu}\left(1_{\hat A}\circ\hat\theta^{m_0}1_B\circ\hat\theta^{q+n-(k-m_0)}
\frac{1}{(2\pi)^2}\int_{[-\pi,\pi]^2}\exp(iu\cdot S_n\psi\circ \hat\theta^q)du\right)\\
&=
\frac{1}{(2\pi)^2}\int_{[-\pi,\pi]^2} 
{\Bbb E}_{\hat\nu}\left(1_{\hat A}\circ\hat\theta^{m_0}1_B\circ\hat\theta^{q+n-(k-m_0)}\exp(iu\cdot S_n\psi\circ\hat\theta^q)\right)du
\end{split}
\]
We then estimate the expectation $a(u)={\Bbb E}_{\hat\nu}(\cdots)$.
Using the fact that the Perron-Frobenius $P$ is the dual of $\hat\theta$ we get
\[
\begin{split}
a(u) &= {\Bbb E}_{\hat\nu}\left( P^q(1_{\hat A}\circ\hat\theta^{m_0})\exp(iu\cdot S_n\psi)1_B\circ\hat\theta^{n-(k-m_0)}\right)\\
&={\Bbb E}_{\hat\nu} \left( P_u^n(P^q(1_{\hat A}\circ\hat\theta^{m_0})1_B\circ\hat\theta^{n-(k-m_0)}\right)\\
&={\Bbb E}_{\hat\nu} \left( P_u^{k-m_0}(1_BP_u^{n-(k-m_0)} P^q(1_{\hat A}\circ\hat\theta^{m_0}))\right).
\end{split}
\]
Let us denote for simplicity $\ell=n-(k-m_0)$.
We first show that for large $u$, the quantity $a(u)$ is negligeable.
Using the contraction inequality given in
proposition \ref{prop:contract} applied to ${P_u}^\ell({\bf 1})$,
the fact that $\Vert P^q({\bf 1}_{\hat A}
\circ\hat\theta^{m_0})\Vert_{\B}
  \le 1+e^{\eta (q+m_0)}$,
and the fact that $\vert{\Bbb E}_{\hat\nu}
[{P_u}^{k-m_0}({\bf 1}_Bg)]\vert
\le \nu_+(B)\Vert g\Vert_{\B}$, we get whenever $u\not\in[-\beta,\beta]^2$,
\begin{equation}\label{eq:1}
|a(u)| \le {\Bbb E}_{\hat\nu}\left({\bf 1}_BP^{\ell}P^q({\bf 1}_{\hat A}\circ \hat\theta^{m_0})\right)
= O( \hat\nu(B) \alpha^{\ell} e^{\eta q}).
\end{equation}

We then estimate the main term, coming from small values of $u$.
The decomposition given in Theorem~\ref{prop:perturb}
gives for any $u\in[-\beta,\beta]^2$
\[
a(u) = \underbrace{ \lambda_u^{\ell} \phi_u
  (P^q({\bf 1}_{\hat A}\circ\hat\theta^{m_0}))
  {\Bbb E}_{\hat\nu}[{P_u}^{k-m_0}({\bf 1}_Bv_u)]}_{a_1(u)}
+
\underbrace{
{\Bbb E}_{\hat\nu}[P_u^{k-m_0}({\bf 1}_B N_u^{\ell}(P^q({\bf 1}_{\hat A}\circ\hat\theta^{m_0})))]
}_{a_2(u)}
\]
Notice that the second term is, by inequality 
\eqref{en:rest} in Proposition~\ref{prop:perturb}, of order 
\begin{equation}\label{eq:2}
a_2(u)=O\left(\hat\nu(B)\alpha^{\ell}e^{\eta q}\right).
\end{equation}
Moreover, since $u\mapsto v_u$ and $u\mapsto \phi_u$ are $C^1$-regular with $v_0=1$ and $\phi_0=\hat\nu$, the first term has the estimate
\[
\begin{split}
a_1(u) &= \lambda_u^{\ell} \hat\nu(\hat A) 
  {\Bbb E}_{\hat\nu}[{P_u}^{k-m_0}({\bf 1}_B)] + O\left( \lambda_u^{\ell}|u| \hat\nu(B) e^{\eta q}\right)\\
&= \lambda_u^{\ell} \hat\nu(\hat A) \hat\nu(B) + O\left( \lambda_u^{\ell}|u| \hat\nu(B)ke^{\eta q}\right)
\end{split}
\]
where the second estimate is obtained by reintroducing the unperturbed Perron-Frobenius operator $P$ in $P_u$, $|{\Bbb E}_{\hat\nu}[{P_u}^{k-m_0}({\bf 1}_B)]-\hat\nu(B)|
= |{\Bbb E}_{\hat\nu}( (e^{iu\cdot S_{k-m_0}\psi}-1){\bf 1}_B)|\le |u|(k-m_0)\|\psi\|_\infty\hat\nu(B)$.

In addition, the intermediate value theorem yields, using $C^3$ smoothness of $\lambda_u$ and Theorem 1 (the bounds \ref{en:bound} and initial values \ref{en:initial}) 
\[
\left|\lambda_u^\ell-\exp(-\frac\ell2\sigma^2_\varphi u\cdot u)\right|
\le \ell (\exp-c_1|u|^2)^{\ell-1} |\lambda_u-\exp(-\frac12\sigma^2_\varphi\cdot u)|
\le C_0\ell e^{-c_1\ell |u|^2}|u|^3 = O\left(e^{-c_2\ell|u|^2}|u|\right)
\]
for the constant $c_2=c_1/2$. Thus 
\[
a_1(u) = \exp(-\frac\ell2\sigma^2_\varphi u\cdot u) \hat\nu(\hat A)\hat\nu(B) + O\left(e^{-c_2\ell|u|^2}|u| \hat\nu(B)ke^{\eta q}\right).
\]
By the classical change of variable $v=u\sqrt{\ell}$ and gaussian integral one easily see that
\[
\int_{[-\beta,\beta]^2}\exp(-\frac\ell2\sigma^2_\varphi u\cdot u)du
=\frac1\ell \int_{[-\beta\sqrt\ell,\beta\sqrt\ell]^2}\exp(-\frac12\sigma^2_\varphi v\cdot v)dv
=\frac{2\pi}{\ell\sqrt{\det\sigma^2_\varphi}} + O\left(\frac{1}{\ell^{3/2}}\right).
\]
Proceeding similarly with the error term one gets as well
\[
\int_{[-\beta,\beta]^2}|u|e^{-c_2\ell|u|^2}du = \frac{1}{\ell^{3/2}}\int_{[-\beta\sqrt{\ell},\beta\sqrt{\ell}]^2}e^{-c_2|v|^2}dv = O\left(\frac{1}{\ell^{3/2}}\right).
\] 
Combining these two computations gives by integration of the approximation of $a_1(u)$ obtained above that
\[
\int_{[-\beta,\beta]^2} a_1(u) du = \frac{2\pi}{\ell\sqrt{\det\sigma^2_\varphi}}\hat\nu(\hat A)\hat\nu(B) + O\left(\frac{\hat\nu(B)ke^{\eta q}}{\ell^{3/2}}\right).
\]
{}From this main estimate and \eqref{eq:1} and \eqref{eq:2} it follows immediately that 
\[
\frac{1}{(2\pi)^2}\int_{[-\pi,\pi]^2}a(u)du = \frac{1}{2\pi\ell\sqrt{\det\sigma^2_\varphi}}\hat\nu(\hat A)\hat\nu(B) + O\left(\frac{\hat\nu(B)ke^{\eta q}}{(n-k)^{3/2}}\right).
\]
\end{proof}

\subsection{Proof of the pointwise convergence of the recurrence rate to the dimension}
\label{sub:zxpointwise}

Let us denote by $G_n(\varepsilon)$ the set of points for which $n$ is an $\varepsilon$-return~:
$$G_n(\varepsilon):= \left\{ x\in\Sigma :
S_n\varphi(x)=0
\ \mbox{and}\ d(\theta^n(x),x)<\varepsilon\right\}.$$
Let us consider the first return time in a $\varepsilon$-neighborhood of a starting point $x\in\Sigma$~:
$$\tau_\varepsilon(x):=
\inf\left\{m\ge 1\ :
S_m\varphi(x)=0
\ \mbox{and}\ d(\theta^m(x),x)<\varepsilon\right\}=\inf\{m\ge1\ : x\in G_m(\varepsilon)\}.$$

\begin{proof}[Proof of Theorem \ref{thm:cas3}]
Let us denote by ${\C}_k$ the set of $k$-cylinders of $\Sigma$.
For any $\delta>0$ denote by ${\C}_k^\delta \subset {\C}_k$ the set of cylinders $C\in {\C}_k$ such that $\nu(C)\in(e^{-(d+\delta)k},e^{-(d-\delta)k})$.
For any $x\in\Sigma$ let $C_k(x)\in {\C}_k$ be the $k$-cylinder which contains $x$. Since $d$ is twice\footnote{Note that we are working with the two-sided symbolic space $\Sigma$.} the entropy of the ergodic measure $\nu$, by the Shannon-McMillan-Breiman theorem, the set 
$K_N^\delta=\{x\in\Sigma\colon \forall k\ge N, C_k(x)\in {\C}_k^\delta\}$
has a measure $\nu(K_N^\delta)>1-\delta$ provided $N$ is taken sufficiently large.

* Let us prove that, almost surely~:
$$\liminf_{\varepsilon\rightarrow 0}{\log\log
\tau_\varepsilon\over-\log\varepsilon}\ge d. $$
Let $\alpha>\frac1{d}$ and $0<\delta<d-\frac{1}{\alpha}$.
Let us take $\varepsilon_n:=\log^{-\alpha}n$ and 
$k_n:=\left\lceil
   -{\log\varepsilon_n}\right\rceil$.
According to Proposition~\ref{prop:pro2}, whenever $k_n\ge N$ we have~:
\[
\begin{split}
\nu(K_N^\delta \cap G_n(\varepsilon_n)) 
&=\nu\left(\{x\in K_N^\delta\colon S_n\varphi(x)=0
\ \mbox{and}\ \theta^n(x)\in C_{k_n}(x)\}\right)\\
&=\sum_{C\in{\C}_{k_n}^\delta}
\nu\left(C\cap\{S_n\varphi=0\}\cap \theta^{-n}\theta^{k_n}(\theta^{-k_n}C)\right)\\
&=\sum_{C\in{\C}_{k_n}^\delta}
\left[{\nu(C)\nu(C)\over n}+O\left({\nu(C)k_ne^{\eta k_n}\over n^{3/2}}\right)
\right].
\end{split}
\]
Observe that for $C\in {\C}_{k_n}^\delta$ we have 
\[
\frac{k_ne^{\eta k_n}}{\sqrt{n}} = \frac{\alpha\log\log n \log^{\alpha\eta}n}{\sqrt{n}} = O(\varepsilon_n^{d+\delta}) = O(\nu(C)),
\]
hence it follows that
\[
\nu(K_N^\delta \cap G_n(\varepsilon_n)) =O\left(\sum_{C\in{\C}_{k_n}^\delta}
{\nu(C)^2\over n}\right)=
O\left({1\over n(\log n)^{(d-\delta)\alpha}}\right)
\]
Hence, by a Borel Cantelli argument, for a.e. $x\in K_N^\delta$, if $n$ is 
large enough, we have~: $\tau_{\varepsilon_n}(x)>n$.
This readily implies that~:
$$\liminf_{n\to\infty}{\log\log
\tau_{\varepsilon_n}\over-\log{\varepsilon_n}}\ge 
    {1\over\alpha}\quad a.e.,$$ 
which proves the lower bound on the $\liminf$ since
$(\varepsilon_n)_n$ decreases to zero
and $\lim_{n\rightarrow +\infty}
{\varepsilon_n\over\varepsilon_{n+1}}=1$.

* Let us prove that, almost surely~:
$$\limsup_{\varepsilon\rightarrow 0}{\log\log
\tau_\varepsilon\over-\log\varepsilon}\le d. $$
Let $0<\alpha<\frac1{d}$ and $\delta>0$ such that $1-\alpha d-\alpha\delta>0$.
Let us take $\varepsilon_n:=\log^{-\alpha}n$ and 
$k_n:=\left\lceil
   -{\log\varepsilon_n}\right\rceil$.
For all $\ell=1,...,n$, we define~:
$$
   A_\ell(\varepsilon)
  := G_\ell(\varepsilon)\cap\theta^{-\ell}\{\tau_\varepsilon>n-\ell\}$$
Let us take $L_n:=\lceil\log^an\rceil$, with
$a>2\alpha(d+\delta+\eta)$. The sets $A_\ell(\varepsilon)$ are 
pairwise disjoint thus~:
\[
1\ge\sum_{\ell=0}^n\nu(A_{\ell}(\varepsilon_n))\ge\sum_{\ell=L_n}^n\sum_{C\in{\C}_{k_n}^\delta}\nu(C\cap A_\ell(\varepsilon_n)).
\]
According to Proposition~\ref{prop:pro2}, we have
\[
\begin{split}
\nu(C\cap A_\ell(\varepsilon_n))
&=\nu\left(C\cap\{S_\ell\varphi=0\}\cap\theta^{-\ell}
\left( C\cap\{\tau_\varepsilon>n-\ell\}\right)\right)\\
&=\left[ \frac{\nu(C)}{2\pi\sqrt{\det\sigma^2}}+
O(\frac{k_n e^{\eta k_n}}{\sqrt{\ell-k_n}})
\right] 
\frac{1}{\ell-k_n}\nu\left(C\cap\{\tau_\varepsilon>n\}\right)\\
&\ge c \varepsilon_n^{d+\delta}\frac{1}{\ell-k_n}\nu\left(C\cap\{\tau_\varepsilon>n\}\right)
\end{split}
\]
for any $C\in{\C}_{k_n}^\delta$ provided $k_n\ge N$; 
indeed, the error term is negligible since~:
$${k_ne^{\eta k_n}\over\sqrt{\ell-k_n}}
  =O\left({(\log\log n)\log^{\alpha\eta}n}\over 
\log^{a/2}(n)\right)=o({\varepsilon_n}^{d+\delta}), $$
since $a>2\alpha(d+\delta+\eta)$. This chain of inequalities gives
\[
\nu(K_N^\delta\cap\{\tau_\varepsilon>n\})
\le 
\sum_{C\in{\C}_{k_n}^\delta}\nu(C\cap\{\tau_\varepsilon>n\})
\le\left(\varepsilon_n^{d+\delta}\log\frac{n-k_n}{L_n-k_n}  \right)^{-1}
=O\left(\frac{1}{\log^{1-\alpha d-\alpha\delta}n}\right).
\]
Now let us take $n_p:=\lfloor 
\exp(p^{2/(1-\alpha d-\alpha\delta)})\rfloor$. We have~:
$$\sum_{p\ge 1}\nu(K_N^\delta\cap\{\tau_{\varepsilon_{n_p}}
    >n_p\})<+\infty. $$
Hence, by the Borel-Cantelli lemma, almost surely $x\in K_N^\delta$,
 we have~:
$$\limsup_{p\rightarrow +\infty}{\log\log\tau
    _{2\varepsilon_{n_p}}\over -\log\varepsilon_
     {n_p}}\le{1\over\alpha}. $$
This gives the estimate $\limsup$ since
$(\varepsilon_{n_p})_p$ decreases to zero
and since $\lim_{p\rightarrow +\infty}
{\varepsilon_{n_p}\over\varepsilon_{n_{p+1}}}=1$.
\end{proof}

\subsection{Fluctuations of the rescaled return time}\label{sub:zxfluctuation}

Recall that $C_k(x)=\{y\in \Sigma\colon d(x,y)<e^{-k}\}$.
Let $R_k(y)=\min\{n\ge1\colon \theta^n(y)\in C_k(y)\}$ denotes the first return time of a point $y$ into its $k$-cylinders $C_k(y)$, or equivalently the first repetition time of the first $k$ symbols of $y$. There have been a lot of studies on this quantity, among all the results we will use the following.
\begin{prop}[Hirata \cite{Hirata}]\label{pro:expstat}
For $\nu$-almost every point $x\in\Sigma$, the return time into the cylinders $C_k(x)$
are asymptotically exponentially distributed in the sense that
\[
\lim_{k\to\infty} \nu_{C_k(x)}\left(R_k(\cdot)>\frac{t}{\nu(C_k(x))}\right) = e^{-t}
\]
for a.e. $x$, where the convergence is uniform in $t$.
\end{prop}

\begin{lem}\label{lem:loi}
Let $x$ be such that $\lim_{k\to\infty} \nu_{C_k(x)}\left(R_k(\cdot)>\frac{t}{\nu(C_k(x))}\right) = e^{-t}$ for all $t>0$.
Then, for all $t>0$, we have~:
$$\lim_{k\rightarrow +\infty}
  \nu\left(\tau_{e^{-k}}>\exp\left.\left(
   {t\over\nu(C_k(x))}\right)\right\vert C_k(x)\right)
    ={1\over 1+\beta t}, $$
with $\beta:={1\over 2\pi\sqrt {det\sigma^2}}$.
\end{lem}
\begin{proof}
We are inspired by the method used by Dvoretzky and
Erd\"os in \cite{DvoretskiErdos}.
Let $k\ge m_0$ and $n$ be some integers.
We make a partition of a cylinder $C_k(x)$ according to the value $\ell\le n$ of the last passage in the time interval $0,\ldots,n$ of the orbit of $(x,0)$ by the map $F$ into $C_k(x)\times\{0\}$. This gives the following equality~:
\begin{equation}\label{sec:formule}
\nu(C_k(x))=\sum_{\ell=0}^n
   \nu\left(C_k(x)\cap\{S_\ell=0\}\cap
    \theta^{-\ell}\left(C_k(x)\cap\{\tau_{e^{-k}}
        >n-\ell\}\right)\right).
\end{equation}

\noindent \underline{Upper bound.} 
Let $n_k=\left\lfloor{e^{t\over\nu(C_k(x))}}\right\rfloor$.
First we claim that~:
$$\limsup_{k\rightarrow +\infty}
  \nu\left(\{ \tau_{e^{-k}}>n_k \}\vert C_k(x)\right)
    \le{1\over 1+\beta t}.$$

Let $a>2\eta$ and $L_k=e^{ak}$.
According to the decomposition~(\ref{sec:formule})
and to Proposition~\ref{prop:pro2}, 
there exists $C'_1>0$ such that we have~:
\begin{eqnarray*}
\nu(C_k(x))&\ge& \nu\left(
C_k(x)\cap\{\tau_{e^{-k}}>n_k \}\right)
 + \sum_{\ell=L_k}^{n_k}
   \beta {\nu(C_k(x))\nu\left(
C_k(x)\cap\{\tau_{e^{-k}}>n_k \}\right)
   \over{\ell-k}}\\
&\ &\ \ \ \ \ -C_1\sum_{\ell=L_k}^{n_k}
   {ke^{\eta k}\nu\left(
   C_k(x)\cap\{\tau_{e^{-k}}>n_k-\ell \}\right)
   \over   (\ell-k)^{3\over 2}}\\
&\ge& \nu\left(
C_k(x)\cap\{\tau_{e^{-k}}>n_k \}\right)
\left(1 +
   \beta \nu(C_k(x))  \sum_{\ell=L_k}^{n_k}
   {1\over{\ell-k}}\right) -C'_1\nu(C_k(x))
   ke^{\eta k}  e^{-ak\over 2}.
\end{eqnarray*}
Hence, we get~:
$${\nu\left(
C_k(x)\cap\{\tau_{e^{-k}}>n_k \}\right)\over
  \nu(C_k(x))}\le  {1- C'_1
   ke^{ k\left(\eta-{a\over 2}\right)} \over
  1 +\beta \nu(C_k(x))  \sum_{\ell=L_k}^{n_k}
   {1\over{\ell-k}}}.$$
The claim follows from the fact that $a>2\eta$.

\noindent \underline{Lower bound.}
Let $b=\liminf\frac{-1}{k}\log\nu(C_k(x))>0$.
Without loss of generality we assume that the H\"older exponent $\eta$ is such that 
$b>2\eta$.
Let $q_k=
\left\lfloor{e^{t\over\nu(C_k(x))}}\right\rfloor$, $n_k=\lfloor q_k\log(q_k)\rfloor$, $m_k=n_k-q_k$ and choose $\delta>0$ such that $2\eta<b(1-\delta)$.
We now claim that~:
$$\liminf_{k\rightarrow +\infty}
  \nu\left(\{\tau_{e^{-k}}>q_k\}\vert C_k(x)\right)
    \ge{1\over 1+\beta t}.$$

Let us denote by $A_\ell(k,x)$ the sets involved in the decomposition (\ref{sec:formule}) :
$$ A_\ell(k,x):=   C_k(x)\cap\{S_\ell=0\}
    \cap
    \theta^{-\ell}\left(C_k(x)\cap\{\tau_{e^{-k}}
        >n_k-\ell\}\right).$$
For $\ell=0$ we have
\begin{equation}\label{sec:EE1}
\nu( A_0(k,x))\le\nu(C_k(x)\cap\{\tau_{e^{-k}}>q_k\}) .
\end{equation}

Let $M_k=\lfloor \nu(C_k(x)^{-1+\delta}\rfloor$. We first show that the contribution from small $\ell$ is negligible. According to the exponential statistics for return times, there exists $\varepsilon_k$, with $\lim_{k\rightarrow +\infty}\varepsilon_k=0$, such that we have (remember that the $A_\ell(k,x)$ are disjoints)~: 
\begin{eqnarray}
\sum_{\ell=1}^{ M_k }\nu(A_\ell(k,x))
&\le& \nu\left( C_k(x) \cap \{ R_k\le M_k \} \right)\nonumber\\
&\le& \nu(C_k(x))\left(1-\exp\left(\nu(C_k(x))^\delta\right)+\varepsilon_k\right)\nonumber\\
&\le& o(\nu(C_k(x))).\label{sec:EE2}
\end{eqnarray}
We now estimate the measure of $A_\ell(k,x)$ for large values of $\ell$.
According to our local limit theorem (Proposition \ref{prop:pro2}), for
all $\ell= M_k+1,\ldots,n_k$, we have~:
\begin{equation}\label{sec:EE3}
\nu(A_\ell(k,x))\le  
\beta \underbrace{
{\nu(C_k(x))\nu\left(
C_k(x)\cap\{\tau_{e^{-k}}>n_k-\ell \}\right)
   \over{\ell-k}}}_{\text{main term}}
   + C_1 \underbrace{   {ke^{\eta k}\nu\left(
   C_k(x)\cap\{\tau_{e^{-k}}>n_k-\ell \}\right)
   \over   (\ell-k)^{3\over 2}}}_{\text{error term}} .
\end{equation}
Observe that the error term is controlled, for some constant $C_2>0$, by 
\begin{equation}\label{sec:EE4}
\sum_{\ell\ge M_k+1 }
      {ke^{\eta k}\nu\left(   C_k(x)\right)
   \over   (\ell-k)^{3\over 2}}
   \le  C_2  \nu(C_k(x))k
   e^{\eta k} (\nu(C_k(x))^{-1+\delta}-k)^{-\frac12} =o(\nu(C_k(x))).
\end{equation}
On the other hand the main term may be estimated for non extremal values of $\ell$ by~:

\begin{equation}
\sum_{\ell=M_k+1 }^{m_k}
{\nu(C_k(x))\nu\left(
C_k(x)\cap\{\tau_{e^{-k}}>n_k-\ell \}\right)
   \over{\ell-k}} \le 
\nu(C_k(x))\nu\left(
C_k(x)\cap\{\tau_{e^{-k}}>q_k \}\right)
   \sum_{\ell=M_k+1 }^{m_k}
  {1\over{\ell-k}}\label{sec:EE5}
\end{equation}
and for extremal values of $\ell$ the simple bound below holds~:
\begin{eqnarray}\label{sec:EE6}
\sum_{\ell=m_k+1}^{n_k}
{\nu(C_k(x))\nu\left(
C_k(x)\cap\{\tau_{e^{-k}}>n_k-\ell \}\right)
   \over{\ell-k}} 
&\le&
  \nu(C_k(x))^2 
   \sum_{\ell=m_k+1}^{n_k}{1\over{\ell-k}}\nonumber\\
&\le& C_3 \nu(C_k(x))^2 
   \log\left({n_k-k\over m_k-k}\right)\nonumber\\
&=& o( \nu(C_k(x)) ).
\end{eqnarray}
Using the decomposition \eqref{sec:formule} and putting together formulas (\ref{sec:EE1}),
(\ref{sec:EE2}), (\ref{sec:EE3}),
(\ref{sec:EE4}), (\ref{sec:EE5}),
(\ref{sec:EE6}), we get~:
\[
\begin{split}
\nu(C_k(x))
&\le 
   \nu(C_k(x)\cap\{\tau_{e^{-k}}>q_k\})
  \left(1+\beta\nu(C_k(x))\sum_{\ell=M_k+1}
  ^{n_k}{1\over{\ell-k}}\right)+o(\nu(C_k))\\
&\le
   \nu(C_k(x)\cap\{\tau_{e^{-k}}>q_k\})
  \left(1+\beta\nu(C_k(x))\log n_k\right) + o(\nu(C_k(x)). 
\end{split}
\]
This proves the claim, which achieves the proof of the lemma.
\end{proof}

\begin{proof}[Proof of Theorem \ref{thm:cas3bis}]
Since the exponential statistics of return time holds a.e. by Proposition~\ref{pro:expstat}, Lemma~\ref{lem:loi} applies a.e. and
 by integration, using Lebesgue dominated convergence theorem, we get that
\[
\lim_{k\to\infty}\nu\left(\log\tau_{e^{-k}}(\cdot)>\frac{t}{\nu(C_k(\cdot))}\right)=\frac{1}{1+\beta t}
\]
for all $t\ge0$.
\end{proof}

\begin{proof}[Proof of Corollary \ref{cor:cas3}]
\underline{Nonzero variance}.
Let us write~:
$$Y_k:= {\log\log\tau_{e^{-k}}(\cdot)
     -kd\over\sqrt{k}} .$$
In this case $\nu$ is a Gibbs measure with a non degenerate h\"older potential $h$.
The logarithm of the measure of the $k$-cylinder about $x$ is, up to some constants, given by the birkhoff sum $\sum_{j=-k}^{k}h\circ\sigma^k(x)$ of $h$ on the orbit of $x$. It is well known that 
such sums follow a central limit theorem (e.g. \cite{Bowen}). 
This readily implies that $X_k
   ={\log(\nu(C_k(\cdot))+kd\over \sqrt{k}}$
converges in distribution to a centered gaussian
random variable of variance $2\sigma^2_h$.
It is enough to prove that $Y_k+X_k$ converges
in probability to 0.
This will be true if $Y_k+X_k$ converges in 
distribution to 0. This follows from
Theorem \ref{thm:cas3bis} and from the formula~:
\[
Y_k+X_k=\frac{\log\log\tau_{e^{-k}}(\cdot)
    +\log(\nu(C_k))}{\sqrt{k}}.
\]
\underline{Zero variance}.
In this case the potential is cohomologous to a constant and the measure $\nu$ is the measure of maximal entropy, which is a Markov measure. Denote by $\pi$ the transition matrix and by $p$ the left eigenvector such that $p\pi=p$.
The measure of a cylinder $C_k(x)$ is equal to $p_{x_{-k}}\prod_{j=-k}^{k-1}\pi_{x_{j}x_{j+1}}$. Since the function $\log \pi_{x_0x_1}$ has to be cohomologous to the entropy, the measure of a cylinder $C_k(x)$ simplifies down to 
\[
\nu(C_k(x))=Q_{x_{-k}x_{k}} e^{-(2k+1)\frac{d}{2}},
\]
where $Q=(Q_{ij})$ is a (constant) matrix.
Proceeding as in the proof of Theorem~\ref{thm:cas3bis}, we get that 
\[
\lim_{k\to\infty}\nu\left(e^{-kd}\log\tau_{e^{-k}}>t\right)
=\sum_{i,j\in\A}\lim_{k\to\infty} \int_{\{x_{-k}=i,x_{k}=j\}} \mathbf{1}_{\{e^{-kd}\log\tau_{e^{-k}}>t\}}d\nu
=
\sum_{i,j\in\A} p_i p_j\frac{1}{1+\beta Q_{ij}t}.
\]
\end{proof}

\end{document}